\documentclass[10pt]{amsart}
 
\usepackage[latin9]{inputenc}
\usepackage{color}
\usepackage{cancel}
\usepackage{amsbsy}
\usepackage{amstext}
\usepackage{amsthm}
\usepackage{amssymb}
\PassOptionsToPackage{normalem}{ulem}
\usepackage{ulem}

\makeatletter
 
\numberwithin{equation}{section}
\numberwithin{figure}{section}

 \usepackage{tikz-cd}
\usepackage{fancyhdr}
 
\usepackage{latexsym}
\usepackage{amscd}\usepackage{amsfonts}\usepackage{pstricks}
\usepackage{young}
\usepackage{enumitem}
\usepackage{amsthm}
\usepackage{mathrsfs}
\usepackage[colorlinks,linkcolor=blue,anchorcolor=blue,citecolor=blue]{hyperref}

 \usepackage{ifthen}
\usepackage{longtable}
\usepackage{todonotes}
\usepackage{marginnote}

\renewcommand{\subsection}[1]{\vspace{3mm}\refstepcounter{subsection}\noindent{\bf \thesubsection. #1.} }

\renewcommand{\subsubsection}[1]{\vspace{3mm}\refstepcounter{subsubsection}\noindent{\bf \thesubsubsection. #1.} }

\numberwithin{equation}{section}

\newtheorem{theorem}{Theorem}[section]

\newtheorem{proposition}[theorem]{Proposition}
\newtheorem{conjecture}[theorem]{Conjecture}

\theoremstyle{definition}

\newtheorem{definition}[theorem]{Definition}
\newtheorem{remark*}[theorem]{Remark}

\newtheorem{example}[theorem]{Example}
\newtheorem*{example*}{Example}

\def\PP{\mathbb P}

\def\ord{\operatorname{ord}}

\def\min{\mathop{\mathrm{min}}}

\def\ZZ{\mathbb Z}
\def\PP{\mathbb P}
\def\K{K}

\def\cal{\mathcal }
 \def\ord{\text{ord}}
 
\def\p{\mathbf p}
\def\q{\mathbf q}

\def\gen{\mathfrak g}

\makeatother

\begin{document}
\title[Campana conjecture  for coverings of toric varieties over function fields]{Campana conjecture  for coverings of toric varieties over function fields}

\author{Carlo Gasbarri}
\address{IRMA, UMR 7501,  7 rue René-Descartes, 67084 Strasbourg, France} 
\email{gasbarri@math.unistra.fr}

\author{Ji Guo}
\address{School of Mathematics and Statistics \\ Central South University \\ Changsha  410075 \\ China} 
\email{221250@csu.edu.cn}

 \author{Julie Tzu-Yueh Wang}
\address{Institute of Mathematics\\
 Academia Sinica\\
 6F, Astronomy-Mathematics Building\\
 No. 1, Sec. 4, Roosevelt Road \\
 Taipei 10617\\
 Taiwan}
\email{jwang@math.sinica.edu.tw}

\thanks{2020\ {\it Mathematics Subject Classification}: Primary 11J97; Secondary 14H05 and 11J87}
\thanks{The  second-named author was supported in part by National Natural Science Foundation of China (No. 12201643) and Natural Science Foundation of Hunan Province, China (No. 2023JJ40690).}
\thanks{The third-named author was supported in part by Taiwan's  NSTC grant  113-2115-M-001-011-MY3.}

\begin{abstract}
 We first prove  Vojta's abc conjecture over function fields for  Campana points on projective toric varieties with high multiplicity along the boundary.    As a consequence, we obtain a version of Campana's conjecture on finite  coverings of projective    toric varieties over function fields. 
  \end{abstract}

\maketitle
\baselineskip=16truept

\section{Introduction }

One of the central guiding principles in Diophantine geometry over number fields and function fields is the slogan: {\it Geometry governs arithmetics}. Roughly speaking, this tells us that one should be able to predict the distribution of rational and integral points of a quasi-projective variety using just informations on its geometry over the algebraic closure of the field. This philosophy  is fully confirmed in dimension one, where Siegel's  
 theorem on finiteness of integral points on affine curves and Faltings' theorem on the finiteness of rational points on curves of genus at least two give a satisfactory answer.  In both cases, the (logarithmic) Kodaira dimension provides a complete prediction of the arithmetic behavior.

For higher dimensional varieties, the situation is considerably more  complicated.  A central conjecture reflecting this philosophy is the celebrated Lang-Vojta conjecture, which can be formulated as follows (see \cite[Proposition 15.9]{Vojta}):
\begin{conjecture}\label{Lang-Vojta}
Let $k$ be a number field and $S$ be a finite set of places of $k$ containing all the archimedean places of $k$.
Let $  X$ be a smooth  projective variety over  $k$, and let $D$ be a   normal crossings divisor on $ X$ over $k$.  Let   ${\mathbf K}_X$   be a canonical divisor of $ X$.
If $V:=  X\setminus D$ is of log general type  (i.e.   ${\mathbf K}_X  +D$ is big), then no set of $(D,S)$-integral  points on $X(k)$ is Zariski dense.
\end{conjecture}
The underlying philosophy is that if a variety admits potentially ``a lot" of rational  or integral points, then it must be  ``special" (which, in the one dimensional case means that it should be not of general type). 
The notion of special variety is clarified in Campana's theory of varieties and requires the introduction of new geometrical objects which nowadays are called ``Campana Orbifolds". 
The theory of Campana Orbifolds is not yet completely developed,  but  key notions such as orbifold structures, orbifold morphisms, and orbifold points have been precisely defined. In particular, one can define the notion of a Campana orbifold pair of general type, which allows to define the notion of special varieties.  Roughly speaking, one conjectures that orbifolds of general type should have rational points which are never Zariski dense (with some caveat in the function fields case due to the presence of isotrivial varieties),  and a variety with potentially dense rational points should not have a morphism to an orbifold of general type. 

One of the most active areas of research in Diophantine geometry is the verification of this philosophy for nontrivial classes of varieties and Campana orbifolds. In particular, Conjecture~\ref{Lang-Vojta} can be viewed as a special case of Campana's more general orbifold conjecture, which is formulated in the function field setting as Conjecture~\ref{CampanaC}.  The function field analogue (in characteristic zero)  of  Conjecture \ref{Lang-Vojta} was established in \cite{GSW2022} for varieties of general type that arise as ramified covers of $\mathbb G_m^n$ with $n\ge 2$.  This work builds upon and extends earlier results of Corvaja and Zannier in the split case  for $n=2$(see \cite{CZ2008} and  \cite{CZ2013}), as well as the non-split surfaces cases by Turchet  \cite{Tur} and by  Capuano and Turchet  \cite{CT}.

In this article, we further verify this claim over function fields of characteristic zero by proving new cases of the Campana conjecture for orbifolds of general type arising as finite coverings of toric varieties.

Let ${\bf k}$ be an algebraically closed field of characteristic zero,
$C$ be a smooth projective curve of genus $\mathfrak{g}$ defined over ${\bf k}$,
and $K:={\bf k}(C)$ be the function field of $C$.  
 At each point $\p\in C(\mathbf{k})$, we may   define a normalized order function $v_{\p}:=\ord_{\p}:\K\to\ZZ\cup\{+\infty\}$.
Let $S$ be a finite set of points of $C$. Denote by ${\cal O}_{S}$ and
${\cal O}_{S}^*$ the sets  of $S$-integers and $S$-units in $K$ respectively.

We now introduce the notion of Campana orbifold and of Campana points:

\begin{definition}\label{campanapair}
A smooth Campana orbifold over $K$ is a pair $(X,\Delta)$ consisting of a smooth projective variety $X$ and an effective $\mathbb Q$-divisor $\Delta$ on $X$, both defined over $K$, such that: 

\begin{enumerate}
\item 
$
\Delta:=\Delta_{\epsilon}=\sum_{\alpha\in \mathcal A_{\epsilon}} \epsilon_{\alpha}D_{\alpha},
$
where the $D_{\alpha}$ are prime divisors on $X$, and $\epsilon_{\alpha}\in \{1-\frac1m: m\in \mathbb Z_{\ge 2}\}\cup\{1\}$ for all $\alpha\in\mathcal A$.
\item The support $\Delta_{\rm red}=\sum_{\alpha\in\mathcal A_{\epsilon}}D_{\alpha}$ is a divisor with  normal crossings on $X$.
\end{enumerate}
\end{definition}

The {\it canonical divisor} of the Campana orbifold $(X,\Delta)$ is the $\mathbb Q$-divisor $ {\mathbf K}_X+\Delta$ on $X$, where ${\mathbf K}_X$ is the canonical divisor of $X$. The Campana orbifold $(X,\Delta)$ is said to be {\it of general type} if $ {\mathbf K}_X+\Delta$  is a big divisor on $X$.

We can now give the definition of {\it Campana point} of a Campana orbifold.

\begin{definition}\label{campdefi}
With the notation introduced above, we say that $\mathcal R\subseteq X(K)$ is a set of  {\it Campana $(\Delta,S)$-integral points} if  (i) no point $P\in \mathcal R$ lies in ${\rm Supp}(\Delta)$, and (ii)  for each $D_{\alpha}$ there is a Weil function $\{\lambda_{D_{\alpha},\p}\}_{\p\in C({\bf k})}$, possibly  depending on $\mathcal R$,  such that the following holds:
\begin{enumerate}
\item[(a)] For all $\alpha$ with $\epsilon_{\alpha}=1$ and $\p\notin S$, one has
$\lambda_{D_{\alpha},\p}(P)=0$.

\item[(b)] For $\p\notin S$, and all $\alpha\in \mathcal A_{\epsilon}$ with both $\epsilon_\alpha<1$ and $\lambda_{D_{\alpha},\p}(P)>0$, we have 
$$
\lambda_{D_{\alpha},\p}(P)\ge \frac 1{1-\epsilon_{\alpha}}. 
$$
\end{enumerate}
\end{definition}
In other words, writing $\epsilon_{\alpha}=1-\frac{1}{m_{\alpha}}$, we require  that $\lambda_{D_{\alpha},\p}(P)\ge m_{\alpha}$ whenever $\lambda_{D_{\alpha},\p}(P)>0$.

\begin{remark*}
If $\epsilon_{\alpha} = 1$ for each divisor $D_{\alpha}$, then the set of Campana $(\Delta, S)$-integral points on $X(K)$ coincides with the set of classical $(D, S)$-integral points on $X(K)$.
\end{remark*}

 Campana's orbifold conjecture further generalizes the Lang-Vojta conjecture (Conjecture \ref{Lang-Vojta}) by incorporating multiplicities along the support of the divisors, rather than ignoring them entirely. In the function field setting, it can be formulated as follows. 
We denote by $h_{ A}$  the Weil height function associated with the divisor $A$ (see Section \ref{weilfunctions}).

\begin{conjecture}\label{CampanaC}
 Let \((X,\Delta)\) be a smooth Campana orbifold of general type.  Let $A$ be a big divisor on $X$. 
Then, for any set $\mathcal R$  of Campana $(\Delta,S)$-integral points,  there exists a proper closed subvariety $Z\subset  X $  such that we have  $h_A(P)\le O(1)$ for all $P\in \mathcal R\setminus Z$.
\end{conjecture}

In order to describe the results of this paper we need another definition:

\begin{definition}\label{campdefimult} Let $\ell $ be a positive integer. A Campana orbifold $(X,\Delta)$ is said to be {\it of multiplicity at least $\ell$}  if $\Delta=\sum_i\epsilon_i\Delta_i$ and   ${{1}\over{1-\epsilon_i}}\geq \ell $ for every $i$.\end{definition}

Following Levin in \cite{Levin:GCD}, we say that an  ``admissible pair"  is a couple $(X,V)$ where  $V$ is a nonsingular variety embedded in a nonsingular projective variety $X$, both defined over $K$,  in such a way that $D_0 =X\setminus V$ is a normal crossings divisor (see \cite[Theorem 2]{GSW2022}).  A Campana orbifold $(X,\Delta)$ is said to be {\it associated to the admissible pair $(X,V)$} if  $D_0 $ is the support of $\Delta$.    
We say that a subset
$Y \subset X$ is in general position with the boundary $D_0 =X\setminus V$  if $Y$ does not contain
any point of intersection of $n$ distinct irreducible components of $D_0 $ (where
$n = \dim X$).

  \begin{definition} We will say that a variety $X$ of dimension $n$ is {\it an admissible toric variety}, if it is a smooth projective toric variety and the couple $(X,\mathbb G_m^n)$ is an admissible pair.
\end{definition}
 
\begin{remark*} In \cite{Fulton} Section 2.5, it is proven that if $X$ is a smooth toric surface, then $(X,\mathbb G_m^2)$ is an admissible pair.
\end{remark*} 

To state the first main result of this paper, we introduce the following notation.
For an effective divisor $D$, we denote by $N_{D,S} $  the counting function and by $N_{D,S}^{(1)}$  the  truncated counting  function with respect to $D$ and $S$ (see Section \ref{weilfunctions} for details).  
 \begin{theorem}\label{toric}
Let  $X$ be an admissible toric  variety of dimension $n\ge 2$. 
Let $D$ be an effective reduced  divisor on $X$ whose support is in general position with the boundary $X\setminus \mathbb G_m^n$.   Let  $A$ be a big divisor on $X$.
Then, for every  $\epsilon >0$, there exists a positive integer $\ell $ such that the following holds:  For every Campana orbifold $(X,\Delta)$   associated to the pair  $(X,\mathbb G_m^n)$ with multiplicity at least $\ell$,  there exists  a proper Zariski closed subset $Z$ of $X$  such that for every set $\mathcal R$ of $(\Delta, S)$-integral points,    either $ h_A(P)\le O(1)$ or  
\begin{enumerate}
\item[{\rm (a)}]
$N_{D,S} (P)-N_{D,S}^{(1)}(P)<\epsilon h_A(P) $, and 
\item[{\rm (b)}] $N_{D,S}^{(1)}(P)\ge h_D(P)-\epsilon h_A(P)-O(1),$ 
\end{enumerate}
  for any $P\in \mathcal R\setminus Z$.
\end{theorem}

As a consequence we obtain the Campana conjecture for some orbifolds whose orbifold divisor contains properly the boundary of a toric variety.

 \begin{theorem}\label{toricGG}
Let  $X$ be an admissible toric  variety of dimension $n\ge 2$.   Let  $D_0:=X\setminus \mathbb G_m^n$  and $A$ be an effective  reduced divisor on $X$ whose support is in general position with $D_0$.   
Then, there exist a positive integer $\ell_0$  and a proper closed set $Z\subset X$ for which the following holds: 
 
 Given a Campana orbifold $(X,\Delta)$   of general type with $Supp(\Delta)=D_0+A$ with multiplicity at least $\ell_0$ along $D_0 $, then
for any set $ \mathcal R\subset X(K)$ of $(\Delta, S)$ integral-points we have  $h_A(P)\le  O(1) $ for any $P\in  \mathcal R\setminus Z$.
\end{theorem}

 We point out that, in the proof we will show that, under the hypothesis of the theorem, the divisor $A$ is big.

 Theorem \ref{toricGG} may be used to give solutions to some explicit diophantine  problems over function fields: 
 \begin{example}  Let $F(x,y,z)\in {\bf k}[x,y,z]$  be a homogeneous polynomial of degree $d\geq 1$. We suppose that the plane curve $[F(x,y,z)=0]$ is smooth and that  $F$ does not vanish at $(1,0,0)$, $(0,1,0)$  and $(0,0,1)$.  Then, we can find  a homogeneous polynomial  $G(x,y,z)$ and a constant $m$ for which the following holds:   Let $\mathcal R_0$ be the set of all non--trivial coprime triples of polynomials $(f_0(t), f_1(t),f_2(t))\in k[t]$.    Let   
 $$\mathcal U_m=\{ (f_0^{n_0}, f_1^{n_1},f_2^{n_2})\,|\, (f_0,f_1,f_2)\in \mathcal R_0, \text{ and each $n_i\ge m$ is an integer }\}.$$
Then, for any tuple in $\mathcal U_m$ such that 
{\rm $F(f_0^{n_0},f_1^{n_1},f_2^{n_2})$ is a {\it perfect power}, }
we have 
$G(f_0^{n_0}(t), f_1^{n_1}(t),f_2^{n_2}(t))=0$.

For the proof, it suffices to apply Theorem \ref{toricGG} to the case where $S=\{0,\infty\}$, $X=\mathbb P^2$,  $D_0$ is the sum of the three coordinate lines,   and  
$\Delta:=(1-\frac1{m})([x_0=0]+ [x_1=0]+ [x_2=0])+\frac1{2}[F=0]$.
\end{example} 
  As a consequence of Theorem \ref{toricGG}, we obtain the following theorem which treats the case of an orbifold whose underlying variety is a finite ramified covering of  an admissible toric  variety:

 \begin{theorem}\label{vojtaconj} 
 Let  $X$ be an admissible toric  variety of dimension $n\ge 2$ and  let $D_0=X\setminus \mathbb G_m^n$.  
Let  $Y$ be a nonsingular projective  variety over $K$ with a finite  morphism $ \pi:Y\to X$.
Let $H:=\pi^{\ast}(D_0)$ and $ R\subset X$ be the ramification divisor of $\pi$,  omitting components from the support of $H$.  Suppose that $A:=\pi(R)$ and $A+D_0$ is a  simple normal crossings  divisor on $X$. Then there exists a positive integer $\ell $ and a proper closed set $Z\subset Y$ such that, if $(Y,\Delta)$ is a Campana orbifold of general type with ${\rm Supp}(\Delta)= {\rm Supp}(H)$ and multiplicity at least $\ell $, then for any set
 $ \mathcal R\subset Y(K)$ of $(\Delta, S)$-integral points, we have either $h_R(P)\le  O(1) $ or $P\in \mathcal R\setminus Z$.
\end{theorem}
 
We will again show that the divisor $R$ is   big  on $Y$.

 The proof of Theorem \ref{toric} relies on the following result, which is a consequence of the main technical theorem in  \cite{GNSW} (see Theorem \ref{main_thmGNSW}).
 To state the result, let $f\in K^*$.  We denote by $N_{0,S}(f)$ and $N_{\infty,S}(f)$  the number of zero and poles, respectively, counted with multiplicity outside of $S$. Similarly,  $N^{(1)}_{0,S}(f)$ and $N^{(1)}_{\infty,S}(f)$  denote the number of distinct zeros and poles of $f$ outside of $S$, i.e., counted without multiplicity.
We also write 
$h(f):=N_{0}(f):=N_{0,\emptyset}(f) $ and $N^{(1)}_{0}(f)=N^{(1)}_{0,\emptyset}(f) $, that is, we take $S=\emptyset$. For a polynomial $F\in K[x_1,\dots,x_n]$, we write $\tilde{h}(F)$ for the relevant height of $F$ (see Section~\ref{Notation and heights} for a precise definition).
\begin{theorem}\label{main_thm}
Let $S$ be a finite  subset of $C$.
Let  $G\in K[ x_1,\hdots, x_n]$ be a non-constant    polynomial  with neither monomial  factors nor repeated factors, and assume that $G(0,\hdots,0)\ne 0$.  Then, for any $\epsilon>0$,   there exist   positive integers $c_0$ and $\ell$, as well as  
a proper Zariski closed subset $Z$ of $\mathbb{A}^n(K)$, such that the following holds:
For every $n$-tuple $(g_1,\hdots,g_n)\in  (K^*)^n\setminus Z$ satisfying 
\begin{align}\label{truncate1}
 N_{0,S}^{(1)}(g_i)+  N_{\infty,S}^{(1)}(g_i) \le \frac1{\ell} h(g_{i}),
\end{align}
and for  each $1\le i\le n$, 
we have either
\begin{align}\label{htubd} 
 \max_{1\le i\le n}\{ h(g_i)\}\le  c_0  \left(\tilde h(G) + \max\{1,2\gen-2+|S|\}\right), 
\end{align}
or  the following two statements hold:
 \begin{enumerate}
 \item[{\rm(a)}]  $N_{0} (G(g_1,\hdots,g_n)  )-N_{0}^{(1)} (G(g_1,\hdots,g_n) )\le \epsilon \max_{1\le i\le n}\{ h(g_i)\}$.
 \item[{\rm(b)}]  If $\deg_{x_i}G=\deg G=d$ for all $1\le i\le n$, then 
   \begin{align*} 
 N_{0,S}^{(1)}\left( G(g_1,\hdots,g_n)\right)\ge  \deg G \cdot (1-\epsilon)\cdot h(1,g_1,\hdots,g_n).
 \end{align*}
  \end{enumerate}
  Moreover, $c_0$ and $\ell$ can be  effectively bounded from above in terms of $\epsilon$, $n$, and the degree of $G$.    The exceptional set $Z$ can be described  as the zero locus of a finite set
$\Sigma\subset K[x_1,\ldots,x_n]$ with the following properties: 
\begin{itemize} 
\item[{\rm(Z1)}] 
  $\Sigma$ depends on $\epsilon$ and $G$ and can be determined explicitly. 
\item[{\rm(Z2)}] The cardinality $\vert \Sigma\vert$ and the degree of each polynomial in $\Sigma$ can be effectively bounded from above in terms of $\epsilon$, $n$, and the degree of $G$.
 \item[{\rm(Z3)}] If  $G\in \mathbf{k}[x_1,\ldots,x_n]$, then  we may take $\Sigma\subset \mathbf{k}[x_1,\ldots,x_n]$.
 \end{itemize}
 \end{theorem}
 
The main technical theorem in  \cite{GNSW} is formulated for $n$-tuples of $S$-units; that is,  the condition \eqref{truncate1} is replaced by $N_{0,S}(g_i)=N_{\infty,S}^{(1)}(g_i)=0$ for all $i$.  
Importantly, the exceptional set $Z$ and the constant $c_0$ and $\ell$ obtained in  \cite{GNSW} are independent of both the choice of the set  $S$ and  the $n$-tuples of $S$-units.  As a consequence, for each  $n$-tuple $(g_1,\hdots,g_n)\in  (K^*)^n$, we are free to enlarge $S$ so that   each $g_i$ becomes  an $S$-unit, where the enlarged set $S$ depends on the given $n$-tuple. 
We then need to control the number of zeros of $G(g_1,\hdots,g_n)$ over this enlarged set $S$, which remains relatively ``small". This viewpoint will be further elaborated in Section \ref{mainmethod}.

 In Section \ref{sec:Preliminaries}, we  recall some basic definitions of heights, local Weil functions associated to divisors and the notion of Campana integral points as well as a generalization of Brownawell-Masser's $S$-unit theorem.  The proof of Theorem \ref{toric} is given in Section \ref{Theorem4}, while the proofs of Theorems \ref{toricGG} and \ref{vojtaconj} are presented in Section \ref{Thm89}.

\section{Preliminaries}\label{sec:Preliminaries}

\subsection{Notation and heights}\label{Notation and heights}
Let ${\bf k}$ be an algebraically closed field of characteristic zero,
$C$ be a smooth projective curve  of genus $\mathfrak{g}$ defined over ${\bf k}$,
and $K:={\bf k}(C)$ be the function field of $C$.
 At each point $\p\in C(\mathbf{k})$,
we may choose a uniformizer $t_{\p}$ to define a normalized order
function $v_{\p}:=\ord_{\p}:\K\to\ZZ\cup\{+\infty\}$.  
Let $S\subset C(\mathbf{k})$ be a finite subset. We denote the ring
of $S$-integers in $K$ and the group of $S$-units in $K$ respectively
by 
\[
{\cal O}_{S}:=\{f\in\K\,|\,v_{\p}(f)\ge0\text{ for all }\p\notin S\}, 
\]
and 
\[
{\cal O}_{S}^{*}:=\{f\in\K\,|\,v_{\p}(f)=0\text{ for all }\p\notin S\}.
\]
We also denote by 
$$
\chi_S(C):= 2\mathfrak{g}-2+|S|, \quad \text{and} \quad \chi_S^+(C):=\max\{0, \chi_S(C)\}.
$$
For simplicity of notation, for $f\in\K^{*}$ and $\mathbf{p}\in C(\mathbf{k})$
we let 
\[
v_{\p}^{0}(f):=\max\{0,v_{\p}(f)\},\quad \text{and}\quad v_{\p}^{\infty}(f):=-\min\{0,v_{\p}(f)\}
\]
i.e. its order of zero and pole at $\p$ respectively.
 The height
of $f$ is defined by 
\[
h(f):=\sum_{\p\in C(\mathbf{k})}v_{\p}^{\infty}(f)=\sum_{\p\in C(\mathbf{k})}v_{\p}^{0}(f).
\]
For a finite subset $S$ of $C(\mathbf{k}) $, $f\in K^{*}$ and  a positive integer $m$, we let 
\[
N_{0,S}({f})=\sum_{\mathbf{p}\in C({\bf k})\setminus S} v_{\p}^{0}(f) \quad \text{and}\quad {N}^{(m)}_{0,S}({f})={\displaystyle \sum_{\mathbf{p}\in C({\bf k})\setminus S}\min\{m,{v}_{\mathbf{p}}^{0}(f)}\}
\]
be the number of the zeros of $f$ outside of $S$, counting multiplicities and counting multiplicities up to $m$ respectively.  We then let
\[
N_{\infty,S}({f})=\sum_{\mathbf{p}\in C({\bf k})\setminus S} v_{\p}^{\infty}(f) \quad \text{and}\quad {N}^{(m)}_{\infty,S}({f})={\displaystyle \sum_{\mathbf{p}\in C({\bf k})\setminus S}\min\{m,{v}_{\mathbf{p}}^{\infty}(f)}\}
\]
be the number of the poles of $f$ outside of $S$, counting multiplicities and counting multiplicities up to $m$ respectively. 
Finally, we let 
\[
N_{0}({f})=\sum_{\mathbf{p}\in C({\bf k})} v_{\p}^{0}(f) \quad \text{and}\quad {N}^{(1)}_{0}({f})={\displaystyle \sum_{\mathbf{p}\in C({\bf k})}\min\{1,{v}_{\mathbf{p}}^{0}(f)}\}
\]
be the number of the zeros of $f$ in $C({\bf k})$ counting multiplicities and counting multiplicities up to $1$ respectively.

Let $\mathbf{x}:=(x_{1},\ldots,x_{n})$ be a tuple of $n$ variables,
and let $F=\sum_{{\bf i}\in I_{F}}a_{{\bf i}}{\bf x}^{{\bf i}}\in K[{\bf x}]$
be a nonzero polynomial, where $I_{F}$ is the set
of those indices ${\bf i}=(i_{1},\hdots,i_{n})$ with $a_{{\bf i}}\ne0$ and
we set ${\bf x}^{{\bf i}}:=x_{1}^{i_{1}}\cdots x_{n}^{i_{n}}$.
We define the height $h(F)$ and the relevant height $\tilde{h}(F)$
as follows. We put 
\begin{align*} 
v_{\p}(F):=\min_{{\bf i}\in I_{F}}\{v_{\p}(a_{{\bf i}})\}\qquad\text{for }  \p\in C({\bf k}),
\end{align*}
and define 
\begin{align*}
h(F):=\sum_{\p\in C({\bf k})}-v_{\p}(F), \qquad\text{and }\quad \tilde{h}(F):=\sum_{  \p\in C({\bf k})}-\min\{0,v_{\p}(F)\}.
\end{align*}

\subsection{Local Weil Functions and Campana integral points}\label{weilfunctions}

We  recall some facts from \cite[Chapter 10]{LangDG} and \cite[Section B.8]{HS} about local Weil functions associated to divisors.   
As before, let $K$ be the function field of a smooth projective curve $C$ over   an algebraically closed field ${\bf k}$  of characteristic zero.  Denote by  $M_K:=\{v=v_{\p}:\p\in C(\mathbf{k})\}$ the set of  valuations on $K$.
We recall that a \emph{$M_K$-constant} is a family $\{\gamma_v\}_{v\in M_K}$, where each $\gamma_v$ is a real number with all but finitely many being zero.  
Given two families $\{\lambda_{1v}\}$ and $\{\lambda_{2v}\}$ of functions parametrized by $M_K$, we say $\lambda_{1v} \le \lambda_{2v}$ holds up to a  $M_K$-constant if there exists a  $M_K$-constant $\{\gamma_v\}$ such that the function $\lambda_{2v} - \lambda_{1v}$ has values at least $\gamma_v$ everywhere.  We say $\lambda_{1v} = \lambda_{2v}$ up to a  $M_K$-constant if $\lambda_{1v} \le \lambda_{2v}$ and $\lambda_{2v} \le \lambda_{1v}$ up to $M_K$-constants. Let $X$ be a projective variety over $K$.  
We say that a subset $Y$ of $X(K)\times M_K$ is {\it affine $M_K$-bounded} if there is an affine open subset $X_0\subset X$   over $K$ with   a system of affine coordinates $x_1,\hdots,x_n$ and a  $M_K$-constant $\{\gamma_v\}_{v\in M_K}$ such that $Y\subset X_0 (K) \times M_K$ and  
$$
\min_{1\le i\le n}v(x_i(P))\ge  \gamma_v, \qquad\text{for all  } (P,v)\in Y.
$$
Finally, we say that the set $Y$ is {\it  $M_K$-bounded}  if it is a finite union of affine $M_K$-bounded sets.

The local Weil functions associated to divisors can be defined geometrically.
We refer to \cite[Section 16]{Vojta} for the following definitions and remarks.
\begin{definition}
Let $X$ be a projective variety defined over ${K}$, which is the function field of a smooth projective curve $C$ over an algebraically closed field ${\bf k}$ of characteristic zero.   
A {\it proper model } of $X$ is a  normal  variety $\mathcal X$, given with a proper flat morphism $\rho:\mathcal X\to C$ such that the generic fiber is isomorphic to $X$.
\end{definition}
In the above construction, rational points in $X(K)$ correspond bijectively  to sections $i:C\to \mathcal X$. 
\begin{definition}\label{model}
Let $X$ be a projective variety defined over ${K}$ and   $\rho:\mathcal X\to C$ be a proper model of $X$.  Let $D$ be a Cartier divisor on $X$ over $K$.  Then $D$ extends to a Cartier divisor $\mathcal D$ on $\mathcal X$.  Let $P\in  X(K)$ not lying on ${\rm Supp} (D)$ and let $i_P:C\to \mathcal X$ be the corresponding section of $\rho$; thus the image of $i_P$ is not contained in $\mathcal D$.  Then $i_P^*\mathcal D$ is a Cartier divisor on $C$.   
Let $\p\in C$ and $n_{D,\p}$ be the multiplicity of $\p$ in  $i_P^*\mathcal D$.
The  \emph{local Weil function} $\lambda_{D,\p}: X(K)\setminus {\rm Supp} (D)  \to \mathbb R$
is defined by
\begin{align}\label{Weilmodel}
\lambda_{D,\p}(P):=n_{D,\p}.
\end{align}
\end{definition}
\begin{remark*}
If we choose another proper model $\mathcal X'$ and we let $\lambda'_{D,\p}$ be the associated local Weil function,
then, $\lambda_{D,\p}-\lambda'_{D,\p}$ is bounded by a $M_K$ constant.  Therefore, we will choose a model that suits the best for our purposes.
\end{remark*} 

For a finite subset $S$ of $C(\mathbf{k})$, we will denote by 
 $$
 m_{D,S}(P):=\sum_{\p\in  S}\lambda_{D,\p}(P), 
$$
$$
N_{D,S}(P):=\sum_{\p\in C(\mathbf{k})\setminus S}\lambda_{D,\p}(P),\quad\text{and}\quad 
N^{(m)}_{D,S}(P):=\sum_{\p\in C(\mathbf{k})\setminus S}\min\{m,\lambda_{D,\p}(P) \},
$$ 
where  $m$ is a positive integer, and by
$$
 h_D(P):=m_{D,S}(P)+N_{D,S}(P).
$$

Finally, we reinterpret the definition of Campana integral points via the geometric model.
Let $(X,\Delta_{\alpha})$ be a Campana orbifold as in Definition \ref{campanapair}, where $\Delta:=\Delta_{\epsilon}=\sum_{\alpha\in \mathcal A_{\epsilon}} \epsilon_{\alpha}D_{\alpha}$.   We can choose a  {\it good integral model away from $S$} which is a proper model $\rho:\mathcal X\to C$ over $\mathcal O_S$ such that $\mathcal X$ is regular.  We denote by $\mathcal D_{\alpha}$ the Zariski closure of $D_{\alpha}$ in $\mathcal X$, and we write $(\mathcal X,\mathcal D_{\epsilon})$ for the model, where $\mathcal D_{\epsilon}:=\sum_{\alpha\in\mathcal A_{\epsilon}}\mathcal D_{\alpha}.$
Then, a rational point $P\in X(K)$ 
extends uniquely to an integral point $\mathcal P\in \mathcal X\mathcal (\mathcal O_S)$.
 
 Following the convention in \cite[Definition 3.4]{PTV2021}, we define the following.
\begin{definition}
With the notation introduced above, we say that $P\in X(K)$ is a {\it Campana $(\Delta,S)$-integral point}  with respect to  $(\mathcal X,\mathcal D_{\epsilon})$ if the following holds:
\begin{enumerate}
\item for all $\alpha$ with $\epsilon_{\alpha}=1$ and $\p\notin S$, 
$\lambda_{D_{\alpha},\p}(P):=n_{D_{\alpha},\p}=0$, i.e. $P\in (X\setminus\cup_{\epsilon_{\alpha}=1} D_{\alpha})(\mathcal O_S)$.

\item for $\p\notin S$, and all $\alpha\in \mathcal A_{\epsilon}$ with both $\epsilon_\alpha<1$ and $n_{D_{\alpha},\p}>0$, we have 
$$
  n_{D_{\alpha},\p}\ge \frac 1{1-\epsilon_{\alpha}}. 
$$
\end{enumerate}
\end{definition}
Consequently, the collection of Campana $(\Delta,S)$-integral points with respect to  $(\mathcal X,\mathcal D_{\epsilon})$ is a set of Campana $(\Delta,S)$-integral points as in Definition \ref{campdefi}.
 
\begin{example}\label{example2}
Let $K$ be the function field of a smooth projective curve $C$ over be an algebraically closed field ${\bf k}$ of characteristic zero.  Let $S\subset C(\mathbf{k})$ be a finite subset. 
Let $X=\mathbb P^2$, $F$ be a non-constant homogeneous polynomial in $\mathcal O_S^*[x_0,x_1,x_2]$  and $D=[F=0]$.  We can take $\mathcal X= C\times \mathbb P^2$.
Let $P=[f_0:f_1:f_2]\in \mathbb P^2(K)$, i.e. $f_i\in K$.  Then, 
$$
\lambda_{D,\p}(P)=v_{\p}(F(f_0,f_1,f_2))-\min\{v_{\p}(f_0),v_{\p}(f_1),v_{\p}(f_2)\} 
$$
for all $\p\notin S$.
Let $\Delta=[x_0=0]+\frac 12[x_1=0]+\frac 23[x_0+x_1+x_2=0]$.
If $P=[f_0:f_1:f_2]\in \mathbb P^2(K)$ is a  Campana $(\Delta,S)$-integral point  with respect to  $(\mathcal X,\mathcal D_{\epsilon})$,  then,  $v_{\p}(f_0)=\min\{v_{\p}(f_0),v_{\p}(f_1),v_{\p}(f_2)\}$  for all $\p\notin S$.  Therefore,  $(\frac{f_1}{f_0}, \frac{f_2}{f_0})\in \mathcal O_S^2$, $v_{\p}(\frac{f_1}{f_0}) \ge 2$, and $v_{\p}(1+\frac{f_1}{f_0}+\frac{f_2}{f_0}) \ge 3$, for all $\p\notin S$.
   \end{example}
 
We will use the following slightly modified result of Brownawell-Masser \cite{BM}.
\begin{theorem}\label{BrMa}  
Let $S$ be a  finite subset of $C({\bf k}).$
If $f_{0},f_{1},\hdots,f_{n}\in K^* $ and $f_0+f_{1}+\cdots+f_{n}=1$,
then either some proper subsum of $f_0+f_{1}+\cdots+f_{n}$ vanishes or
\[
\max_{0\le i\le n}\{h(f_{i})\}\le \sum_{i=0}^n \left(N_{ 0,S}^{(n)}(f_i)+N_{\infty,S}^{(n)}(f_i)\right)+\frac{n(n+1)}{2}\chi_S^+(C).
\]
\end{theorem}

 \section{Proof of Theorem \ref{main_thm}}\label{mainmethod}
The proof of Theorem~\ref{main_thm} relies on the main technical theorem of \cite{GNSW}, which we state below.

 \begin{theorem}[{\cite[Theorem 4]{GNSW}}]\label{GNSW}\label{main_thmGNSW}
Let  $G\in K[ x_1,\hdots, x_n]$ be a non-constant    polynomial  with neither monomial  factors nor repeated factors. 
Then for any $\epsilon>0$,   there exist a positive real number $c_0$  and  
a proper Zariski closed subset $Z$ of $\mathbb{A}^n(K)$ such that for all  $(u_1,\hdots,u_n)\in ({\cal O}_{S}^*)^{n}\setminus Z$, we have either
\begin{enumerate}
\item[{\rm(i)}]   the inequality 
\begin{align}
 \max_{1\le i\le n}\{ h(u_i)\}\le  c_0  \left(\tilde h(G) + \max\{1,2\gen-2+|S|\}\right) 
\end{align}
holds, 
  \item[{\rm(ii)}]   or the following two statements hold:

 \begin{enumerate}
 \item[{\rm(a)}]  $N_{0,S}\left( G(u_1,\hdots,u_n) \right)-N_{0,S}^{(1)}\left( G(u_1,\hdots,u_n)\right)\le \epsilon \max_{1\le i\le n}\{ h(u_i)\}$,\newline  if $G(u_1,\hdots,u_n)\ne 0$.
 \item[{\rm(b)}]  If $G(0,\hdots,0)\ne 0$ and $\deg_{x_i}G=\deg G=d$ for $1\le i\le n$, then 
 \begin{align*} 
 N_{0,S}^{(1)}\left( G(u_1,\hdots,u_n)\right)\ge  \deg G \cdot (1-\epsilon)\cdot h(1,u_1,\hdots,u_n).
 \end{align*}
  \end{enumerate}
  \end{enumerate}
 Here, $c_0$ can be  effectively bounded from above in terms of $\epsilon$, $n$, and the degree of $G$.    Moreover, the exceptional set $Z$ can be expressed as the zero locus of a finite set
$\Sigma\subset K[x_1,\ldots,x_n]$ with the following properties: {\rm(Z1)} $\Sigma$ depends on $\epsilon$ and $G$ and can be determined explicitly,  {\rm(Z2)} $\vert \Sigma\vert$ and the degree of each polynomial in $\Sigma$ can be effectively bounded from above in terms of $\epsilon$, $n$, and the degree of $G$, and  {\rm(Z3)}  if  $G\in \mathbf{k}[x_1,\ldots,x_n]$, then  we may take $\Sigma\subset \mathbf{k}[x_1,\ldots,x_n]$.
 \end{theorem}

We will also need the following result for the proof of Theorem \ref{main_thm}.
\begin{theorem}\label{ProximityAffine} 
Let  $F \in K[x_1,\dots,x_n]$ be a non-constant polynomial, and assume that $F(0,\dots,0)\ne 0$.   Let $W$ be the Zariski closed subset   of $\mathbb A^n(K)$ defined by the vanishing of all non-trivial subsums (including $F$ itself) appearing in the expansion of $F$.  Let $S$ be a finite subset of $C({\mathbf k})$.  Then,
for all $(g_1,\hdots,g_n)\in (\mathcal O_S^*)^n\setminus W$, we have 
 \begin{align*} 
 \sum_{\p\in S} v_{\p}^0(F(g_1,\hdots,g_n)) \le  
\tilde c_1  \chi_S^+(C) +\tilde c_2h(F),
  \end{align*}
  where $d=\deg F$, $\tilde c_1=\frac12 \binom{n+d}{n}  \left(\binom{n+d}{n}+1\right)$ and $\tilde c_2=2 \left(\binom{n+d}{n}-1\right)\tilde c_1$.
   \end{theorem}
 
\begin{proof}
Since $F(0,\hdots,0)\ne 0$ and $h(F)=h(\lambda F)$ for any $\lambda\in K^*$, we may assume that $F(x_1,\hdots,x_n)=1+\sum_{\mathbf i\in I}a_{\mathbf i} \mathbf{x}^{\mathbf i}$, where $\mathbf{i}\in\mathbb{Z}_{\ge0}^{n}$,
 $|\mathbf{i}|\le d$, and $a_{\mathbf i}\ne 0$.  Then, we have
 \begin{align}\label{Fequ}
 1=F( {\mathbf g})-\sum_{\mathbf i\in I}a_{\mathbf i}  {\mathbf g}^{\mathbf i},
 \end{align}
where $ {\mathbf g}=(g_1,\hdots,g_n)\in (\mathcal O_S^*)^n\setminus W$.
Therefore,  no subsum of $\sum_{\mathbf i}a_{\mathbf i} {\mathbf g}^{\mathbf i}$ vanishes.  
Let
 $S'=\{\p\in C\setminus S\,|\, v_{\p}(a_{\mathbf i})< 0 \text{ for some } {\mathbf i\in I} \}\cup S.$  Then   $F( {\mathbf g})\in\mathcal O_{S'}$, and
\begin{align}\label{sizeS'}
 |S'|\le |S|+ c_1 h(F),
\end{align}
 where $c_1=  \binom{n+d}{n}-1$.
Applying Theorem \ref{BrMa} to \eqref{Fequ}  with respect to $S'$ and using \eqref{sizeS'}, we obtain
 \begin{align}\label{htbdd}
 h(F( {\mathbf g})) &\le N_{0,S}(F( {\mathbf u})) 
  +\tilde c_1\left(\max\{0,2\gen-2+|S'|\} \right)\cr
  &\le N_{0,S}(F( {\mathbf u})) 
  +\tilde c_1\left( \chi_S^+(C)+  c_1 h(F)\right), 
\end{align}
where $\tilde c_1=\frac12 (c_1+1)(c_1+2)$.
Since 
$$h(F( {\mathbf g}))=\sum_{\p\in C} v_{\p}^0(F( {\mathbf g}))=  N_{0,S}(F( {\mathbf g})) +\sum_{\p\in S} v_{\p}^0(F( {\mathbf g})),$$
  the inequality \eqref{htbdd} implies
 \begin{align*} 
 \sum_{\p\in S} v_{\p}^0(F( {\mathbf g})) \le  
\tilde c_1 \chi_S^+(C) + c_1\tilde c_1h(F),
 \end{align*}
 as asserted.
\end{proof}

\begin{proof}[Proof of Theorem \ref{main_thm}]
  Let $\ell$ be a sufficiently large integer to be determined later.  Let ${\mathbf g}=(g_1,\hdots,g_n)$, and define
\begin{align*} 
\tilde S_{\mathbf g}=\{\p\in C\setminus S\,|\, v_{\p}( g_i)\ne 0\quad\text{for some $1\le i\le n$}\}\cup S.
\end{align*}
Then, together with \eqref{truncate1}, we obtain
\begin{align}\label{sizeS}
|\tilde S_{\mathbf g}  |\le \sum_{i=1}^n \big( N_{0,S}^{(1)}(g_i)+N_{\infty,S}^{(1)}(g_i) \big) +|S|\le  \frac{n}{\ell} \max_{1\le i\le n}\{h(g_{i})\} +|S|,
\end{align}
 and moreover $g_i\in\mathcal O_{\tilde S_{\mathbf g}}^*$ for  all  $1\le i \le n$.  
Let $\epsilon>0$ be given.  We now apply Theorem \ref{main_thmGNSW} to obtain a  positive real number $\tilde c_0$  and  
a proper Zariski closed subset $Z$ of $\mathbb{A}^n(K)$ such that  if  $(g_1,\hdots,g_n)\notin  Z$,   either 
\begin{align}\label{htgbd}  
 \max_{1\le i\le n}\{ h(g_i)\}&\le   \tilde c_0  \left(\tilde h(G) + \max\{1,2\gen-2+|\tilde S_{\mathbf g}|\}\right)\cr
 &\le \tilde c_0  \left(\tilde h(G)+ \frac{n}{\ell} \max_{1\le i\le n}\{h(g_{i})\} + \max\{1,2\gen-2+|S|\}\right)
\end{align}
by \eqref{sizeS}, or both of the following two inequalities hold:
\begin{align}\label{Ns} 
N_{0,\tilde{S}_{\mathbf{g}}}\left( G(g_1, \dots, g_n) \right) - N_{0,\tilde{S}_{\mathbf{g}}}^{(1)}\left( G(g_1, \dots, g_n) \right) &\leq \epsilon \max_{1 \leq i \leq n} \{ h(g_i) \},
\end{align} 
and
 \begin{align}\label{truncateNs}  
 N_{0,\tilde{S}_{\mathbf{g}}}^{(1)}\left( G(g_1, \dots, g_n) \right) &\geq \deg G \cdot (1-\epsilon) \cdot h(1, g_1, \dots, g_n)
 \end{align}
if we further assume that $\deg_{x_i}G=\deg G=d$ for $1\le i\le n$.

Since the constant $c_0$ and the exceptional set $Z$ in Theorem \ref{main_thmGNSW}  does not depend on $S$ and the arguments (i.e. the $u_i$), it is important to note that they also do not depend on ${\mathbf g}$ in the current context.
Now choose $\ell>2\tilde  c_0n$ and $c_0=2\tilde c_0$.  Then  it follows from \eqref{htgbd} that 
\begin{align}\label{htgbd2}  
 \max_{1\le i\le n}\{ h(g_i)\} \le  c_0  \left(\tilde h(G) + \max\{1,2\gen-2+|S|\}\right).
\end{align}
Next, we aim to show that
\begin{align}\label{NS} 
N_{0}\left( G(g_1,\hdots,g_n) \right)-N_{0}^{(1)}\left( G(g_1,\hdots,g_n)\right)\le 2\epsilon \max_{1\le i\le n}\{ h(u_i)\},
\end{align}
under the assumption that \eqref{htgbd2} does not hold. (If necessary, we may further increase the constant $c_0$ accordingly.)
Let $W$ be the Zariski closed subset   of $\mathbb A^n(K)$ defined by the vanishing of all possible subsums appearing in the expansion of $G$ (including $G$ itself).   Since $G(0,\dots,0)\ne 0$, we can apply    Theorem \ref{ProximityAffine}  to $\mathbf g\in (\mathcal O_{\tilde{S}_{\mathbf{g}}}^*)^n\setminus W$ . 
This gives
 \begin{align*} 
 \sum_{\p\in \tilde S_{\mathbf g}} v_{\p}^0(G(g_1,\hdots,g_n)) \le  
\tilde c_1 \chi_{\tilde S_{\mathbf g}}^+(C) +\tilde c_2h(G),
  \end{align*}
where $d=\deg G$, $\tilde c_1=\frac12 \binom{n+d}{n}  \left(\binom{n+d}{n}+1\right)$ and $\tilde c_2=2  \left(\binom{n+d}{n}-1\right)\tilde c_1$.
Using \eqref{sizeS}, we deduce
 \begin{align}\label{proxi}
 \sum_{\p\in \tilde S_{\mathbf g}} v_{\p}^0( G(g_1,\hdots,g_n)) 
&\le  
\tilde c_1\left(\frac{n}{\ell} \max_{1\le i\le n}\{h(g_{i})\}+\chi_S^+(C)\right)+\tilde c_2h(G)\cr
&\le \epsilon \max_{1\le i\le n}\{h(g_{i})\},
  \end{align}
provided we choose  $\ell>2n\tilde c_1\epsilon^{-1}$ and $c_0\ge  2\tilde c_2\epsilon^{-1}$. 
Thus, combining \eqref{Ns} and \eqref{proxi} yields \eqref{NS} for all $\mathbf g\notin Z\cup W$.

Finally, the proof of (b) follows easily.
Since $S\subset {\tilde S_{\mathbf g}}$, it is clear that 
$$
N^{(1)}_{0,S}( G(g_1,\hdots,g_n) )\ge N_{0,\tilde S_{\mathbf g}}^{(1)}( G(g_1,\hdots,g_n))\ge  \deg G \cdot (1-\epsilon)\cdot h(1,g_1,\hdots,g_n)
$$
  by \eqref{truncateNs}, as wanted.
 \end{proof}

 \section{Proof of Theorem \ref{toric}}\label{Theorem4}
 Before proving Theorem~\ref{toric}, we present the following example to illustrate the strategy of the proof.
Let $K$ be the function field of a smooth projective curve $C$ over be an algebraically closed field ${\bf k}$ of characteristic zero.  Let $S\subset C(\mathbf{k})$ be a finite subset. 
\begin{example}\label{example3}
We consider the pair $(\mathbb{P}^n, \mathbb{G}_m^n)$ as an admissible pair, with boundary divisor $D_0 := H_0 + \cdots + H_n$, where $H_i$ for $0 \le i \le n$ are the coordinate hyperplanes in $\mathbb{P}^n$. Let $\Delta = \ell D_0$ for a sufficiently large integer $\ell$.   Let $F$ be a non-constant homogeneous polynomial in ${\bf k}[x_0,\hdots,x_n]$  of degree $d$ with neither monomial  factors nor repeated factors and $D=[F=0]$.   Assume that $D$ is in general position with the coordinate hyperplanes $H_i$; that is, none of the values ${F(1,0,\ldots,0), \ldots, F(0,\ldots,0,1)}$ vanish.

Let $P = [f_0 : \cdots : f_n] \in \mathbb{P}^n(K)$, where each $f_i \in K$. Denote $\mathbf{f} = (f_0, \ldots, f_n)$, and for each place $\mathfrak{p} \in C({\bf k})$, define 
$$e_{\p}(\mathbf f):=\min\{v_{\p}(f_0),\hdots,v_{\p}(f_n)\}.$$
Then $P $ is an $(\Delta,S)$-integral points if 
$$v_{\p}(f_i)-e_{\p}(\mathbf f)\ge \ell\quad\text{for all }  \p\notin S.$$
For each $0 \le i \le n$, define the set 
$$ S_{P,i}:=  \{\p\notin S:\,  v_{\p}(f_i)>e_{\p}(\mathbf f)\}\cup S.$$ 
Then, 
\begin{align}\label{SizeSi}
| S_{P,i}|\le |S|+\frac{1}{\ell}  (v_{\p}(f_i)-e_{\p}(\mathbf f ))\le |S|+\frac{1}{\ell}\max_{0\le j\le n}\{h(\frac{f_j}{f_i})\}. 
\end{align}
Moreover,
\begin{align}\label{condPn}
 \frac{f_j}{f_{i}} \in\mathcal O_{S_{P,i}}   \quad\text{and}\quad\ N_{0,S_{P,i}}^{(1)}( \frac{f_j}{f_{i}})\le \frac 1{\ell }N_{0,S_{P,i}}( \frac{f_j}{f_{i}})\le \frac 1{\ell }h( \frac{f_j}{f_{i}}).
\end{align}

Now, for all $\mathfrak{p}\in  C(\mathbf{k})$, we compute the local Weil function:
$$
\lambda_{D,\p}(P)=v_{\p}(F(f_0\hdots,f_n))-de_{\p}(\mathbf f).
$$
 Define $T_{P,i}:=\{\p\notin S\,:\,  v_{\p}(f_i)=e_{\p}(\mathbf f)\}$.  Then, 
we can estimate the counting function as follows:
\begin{align*}
N_{D,S}(P) &:= \sum_{\p \notin S} \lambda_{D,\p}(P) = \sum_{\p \notin S} v_{\p}\big(F(f_0, \ldots, f_n)\big) - d e_{\p}(\mathbf{f}) \notag \\
&\le \sum_{i=0}^n \sum_{\p \in T_{P,i}} v_{\p}\left(F\left(\frac{f_0}{f_i}, \ldots, \frac{f_n}{f_i}\right)\right) 
= \sum_{i=0}^n \sum_{\p \notin S_{P,i}} v_{\p}\left(F\left(\frac{f_0}{f_i}, \ldots, \frac{f_n}{f_i}\right)\right).
\end{align*}
Similarly, we obtain:
 \begin{align*}
&N_{D,S}(P) - N^{(1)}_{D,S}(P)\\
&\le \sum_{i=0}^n \sum_{\p \notin S_{P,i}}  
v_{\p}\left(F\left(\frac{f_0}{f_i}, \ldots, \frac{f_n}{f_i}\right)\right)
- \min\left\{1, \, v_{\p}\left(F\left(\frac{f_0}{f_i}, \ldots, \frac{f_n}{f_i}\right)\right)\right\}.
\end{align*}
Combining inequalities \eqref{SizeSi} and \eqref{condPn}, we are now in a position to apply Theorem~\ref{main_thm} for each $i$.
 \end{example}

The proof of Theorem \ref{toric} is based on Theorem \ref{main_thm} and  Theorem \ref{ProximityAffine}.
We will use the following proposition adapting from \cite[Proposition 10.11]{Vojta}.
\begin{proposition}\label{bigdivisor}
Let $X$ be a projective variety over $K$.  Let $A$ and $B$ be two big divisors on $X$. Then, there exist constants $c_1$ and $c_2$ and a proper Zariski closed subset $Z$ of $X$, depending only on $A$ and $B$, such that 
$$
c_1h_A(P)-O(1)\le h_B(P)\le c_2 h_A(P)+O(1)
$$ 
for all $P\in X(K)\setminus Z$, where the implied constant depends only on $A$, $B$ and the choices of the height functions.
\end{proposition}
  We note that the implied constants will be dropped (by increasing $c_1$ and $c_2$) when the height functions are fixed.

   \begin{proof} [Proof of Theorem \ref{toric}]
We fix a proper model $\mathcal X$ of $X$  so that the local Weil functions of divisors of $X$ are defined as in Definition \ref{model}. 
We recall the following  setup  of finding a natural finite open covering of $X$ from the proof of \cite[Theorem 4.4]{Levin:GCD}.
Let $\Sigma$ be the fan corresponding to the smooth projective toric variety $X$.  Then, there is a finite affine  covering $\{X_{\sigma} \}$ of $X$, where $\sigma\in \Sigma$ is a $n$-dimensional smooth cone with an isomorphism $i_{\sigma}: X_{\sigma}\to\mathbb A^n$.  This isomorphism restricts to an automorphism of $\mathbb G_m^n$, where we identify $\mathbb G_m^n\subset X_{\sigma}$ naturally as a subset of $X$ and 
$\mathbb G_m^n\subset\mathbb A^n$ in the standard way such that $\mathbb A^n\setminus \mathbb G_m^n$ consists of the affine coordinate hyperplanes $\{x_i=0\}$, $1\le i\le n$.  Moreover, by Proposition \ref{bigdivisor} {\blue and \cite[Remark 10.12]{Vojta}}, there exist  non-zero constants $b_{\sigma,A}$, $c_{\sigma,A}$ and  a proper closed subset $Z_{\sigma,A}\subset X_{\sigma}$, depending on $\sigma$ and  $A$ such that 
\begin{align}\label{htsigma}
b_{\sigma,A} h_A(P)\le h( {i}_{\sigma}(P))\le c_{\sigma,A} h_A(P)
\end{align}
for all $P\in X_{\sigma}(K)\setminus Z_{\sigma,A}\subset X(K)$,  where $ {i}_{\sigma}(P) =(u_1,\hdots, u_n )$ and $h( {i}_{\sigma}(P)):=\max_{1\le i\le n}\{h(u_i)\}$ (note that we are including the coordinate hyperplane of $X_{\sigma}$ as a subset of $Z_{\sigma,A}$).
The pullback $(i_{\sigma}^{-1})^*(D|_{X_{\sigma}})$ of $D$ to $\mathbb A^n$ is defined by some nonzero polynomial $f_{\sigma}\in K[x_1,\hdots,x_n]$, which does not vanish at the origin since $D$ is in general position with the boundary of $\mathbb G_m^n$ in $X$.  We will take  $f_{\sigma}$ with $f_{\sigma}(0,\hdots,0)=1$.

 Let  $(X,\Delta)$ be an  Campana orbifold associated to $(X,\mathbb G_m^n)$ with multiplicity at least $\ell$, where $\ell$ is a sufficiently large integer to be determined later. 
Let $\mathcal R$ be a set of $(\Delta, S)$-integral points. 
Let  $P\in  \mathcal R\setminus {\rm Supp}(D)$.   By  \cite[Chapter 10, Proposition 1.2]{LangDG}, for each $\q\notin S$ we find  some $\sigma$ (depend on $\q$) such that $P\in X_{\sigma}(K)\setminus Z_{\sigma,A}$ with  $v_{\q}( {i}_{\sigma}(P)) :=\min\{v_{\q}(u_1),\hdots,v_{\q}(u_n) \}\ge 0$.  Then, we have
 \begin{align}\label{Psigma}
\lambda_{D,\q}(P)=v_{\q}^0 (f_{\sigma}(i_{\sigma}(P)))+c_{\sigma,\q}, 
\end{align} 
where $0\le c_{\sigma,\q}\le -v_{\q}(f_{\sigma})$.
We note that $v_{\q}(f_{\sigma})\le 0$ since $f_{\sigma}(0,\hdots,0)=1$.
Since there are only finitely many $X_\sigma$, we find a finite index set $I_P$  such that for any $\p\notin S$ the equality \eqref{Psigma}  holds for some $\sigma\in I_P$.  Let 
$S_{\sigma}:=\{\p\in C({\bf k}): v_{\p}(u_i)< 0, \text{ for some }   1\le i\le n\}\cup S$.
 Then,
\begin{align}\label{cond}
(u_1,\hdots,u_n)\in\mathcal O_{S_{\sigma}}^n \quad\text{and}\quad\ N_{0,S_{\sigma}}^{(1)}(u_i)\le \frac 1{\ell }N_{0,S_{\sigma}}(u_i)\le \frac 1{\ell }h(u_i),
\end{align}
 since $P$ is a  $(\Delta, S)$-integral point with multiplicity at least $\ell$. 
 Furthermore,
\begin{align}\label{countingD}
N_{D,S} (P)-N_{D,S}^{(1)}(P)\le \sum_{\sigma\in I_P}N_{0,S_{\sigma}}   (f_{\sigma}(i_{\sigma}(P)))-N_{0,S_{\sigma}} ^{(1)}  (f_{\sigma}(i_{\sigma}(P)))+O(1), 
\end{align}
by \eqref{Psigma}.
Let $M$ be the number of $X_{\sigma}$; then, $|I_P|\le M$.
Finally, we note that $|S_{\sigma}|\le  |S|+\#\{\p\notin S| \min_{1\le i\le n}\{v_{\p}(u_i)\}<0 \}$. 
Let $\mathcal B$ be the divisor on $\mathcal X$ coming from the boundary $D_0$ and $\iota_P:C\to \mathcal X$ 
be the corresponding section of $P$.  Then,
for $\p\in S_{\sigma}\setminus S$, we have $ \iota_P(\p)\in \mathcal B$,  and 
$$
-\min\{0,v_{\p}(u_1),\hdots, v_{\p}(u_n)\}=n_{B_{\alpha},\p}\ge \ell,
$$ 
where $B_{\alpha}$ is a component of $D_0$.
Therefore, we have
\begin{align}\label{sizeSs}
 |S_{\sigma}|\le |S|+\frac1{\ell} h(1,u_1,\hdots,u_n)\le  |S|+ \frac {n}{\ell} \max_{1\le i\le n}\{h(u_i)\}.
\end{align}

We may now apply Theorem \ref{main_thm} and \eqref{sizeSs} to each $f_{\sigma}$ and  $S_{\sigma}$.  Let $\epsilon>0$.  Then there exist a positive real $b_{\sigma}$ and a proper Zariski closed subset  $W_{\sigma}\subset X(K)$ containing ${\rm Supp} (D)$  and $Z_{\sigma,A}$ for each $\sigma\in \Sigma$  such that for sufficiently large integer $\ell$, either
\begin{align}\label{htbddsigma}
h({i}_{\sigma}(P))\le b_{\sigma}\left(h(f_{\sigma}) +\max\{1,\chi_S(C)\}\right)+ \frac{2b_{\sigma}}{\ell}  h({i}_{\sigma}(P))
\end{align} 
or
\begin{align}\label{gcdsigma}
N_{0,S_{\sigma}}   (f_{\sigma}(i_{\sigma}(P)))-N_{0,S_{\sigma}} ^{(1)}  (f_{\sigma}(i_{\sigma}(P))) 
\le  \frac{\epsilon}{Mc_{\sigma,A}} h({i}_{\sigma}(P)) +O(1)  
\end{align}
for all   
 $P\in  \mathcal R\cap X_{\sigma}(K)\setminus W_{\sigma}$.     
 Since the number of $X_{\sigma}$ is finite, we conclude from \eqref{htsigma}, \eqref{countingD}, \eqref{htbddsigma} and \eqref{gcdsigma} that there exist  a positive integer $\ell_0$ and a proper Zariski closed subset $Z_1$ of $X$  such that either
$h_A(P)\le O(1)$ or
\begin{align}\label{multiplicity}
N_{D,S} (P)-N_{D,S}^{(1)}(P)<\epsilon h_A(P) 
\end{align}
 for all   $ P\in  \mathcal R \setminus Z_1$ if $\ell>\ell_0$.  
 
Let $\q\in S$. Let  $ P\in  \mathcal R\setminus {\rm Supp}(D)$.  We find  some $\sigma$ (depending on $\q$) such that $P\in X_{\sigma}(K)\setminus Z_{\sigma,A}$ and  $v_{\q}( {i}_{\sigma}(P)) :=\min_{1\le i\le n}\{v_{\q}(u_i)  \}\ge 0$.  Then \eqref{Psigma} holds as well.
Let $S_{\sigma,\q}:=\{\p\in C({\bf k}): v_{\p}(u_i)< 0, \text{ for some }   1\le i\le n\}\cup S$.
Therefore, $ {i}_{\sigma}(P) =(u_1,\hdots, u_n )\in (\mathcal O_{S_{\sigma,\q}}^*)^n$.
We then  apply Theorem \ref{ProximityAffine} 
 to  each set $S_{\sigma,\q}$, for each $\q\in S$, to find a Zariski closed subset  $Z_{\sigma,\q}$ containing $Z_{\sigma,A}$ from \eqref{htsigma}  of $X_{\sigma}$ such that
\begin{align}\label{fsigma} 
 v_{\q}^{0} (f_{\sigma}(i_{\sigma}(P))\le \tilde c_1\left( \chi_{S_{\sigma,\q}}^+(C)\right)+ \tilde c_2 h(f_{\sigma}) 
\end{align} 
for all   $  P\in  \mathcal R \cap X_{\sigma}(K)\setminus Z_{\sigma,\q}$.
Using \eqref{sizeSs}, the right hand side of \eqref{fsigma} can be estimated as 
\begin{align*} 
\chi_{S_{\sigma,\q}}^+(C)\le 2\gen+ |S|+ \frac {n}{\ell} \max_{1\le i\le n}\{h(u_i)\}.
\end{align*}
 By repeating the previous arguments used in the proof of \eqref{multiplicity}, and by enlarging $\ell_0$ if necessary, we have either
$h_A(P)\le O(1)$, or 
\begin{align}\label{lambdaq} 
m_{D,S}(P ):=\sum_{\q\in S}\lambda_{D,\q}(P)= \sum_{\q\in S}v_{\q}^{0} (f_{\sigma}(i_{\sigma}(P))+O(1)\le \epsilon h_A(P)
\end{align}
for all  $ P\in  \mathcal R \setminus Z_2$, where $Z_2$ is the union of all $Z_{\sigma,\q}$  for $\q\in S$ if $\ell\ge\ell_0$.   
 We note that, according to Definition \ref{model} and its accompanying remark, the Zariski closed subsets $Z_1$ and $Z_2$ remain unchanged regardless of the set $\mathcal{R}$ of $(\Delta, S)$-integral points, while the implied constant $O(1)$ may vary.  
 
Since $h_D(P)= m_{D,S}(P )+N_{D,S}(P )$, we can derive from   \eqref {lambdaq} that
$$
N_{D,S} (P)\ge h_D(P)- \epsilon h_A(P) 
$$
 for all   $ P\in \mathcal R \setminus Z_2$.
Then we have either $h_A(P)\le O(1)$  or 
\begin{align}\label{final}
N_{D,S}^{(1)}(P)\ge h_D(P)- 2\epsilon h_A(P)-O(1), 
\end{align}
for all    $P\in \mathcal R \setminus \{Z_1\cup Z_2 \}$.
Finally, we note that the situation that $h_A(P)\le O(1)$ can be included in \eqref{final} by enlarging the implied constant.
\end{proof}
 
 \section{Proof of Theorem \ref{toricGG} and Theorem \ref{vojtaconj}}\label{Thm89}

  \begin{proof}[Proof of Theorem \ref{toricGG}]
 Since a canonical divisor ${\mathbf K}_X$ can naturally be taken  to be $-D_0$ (see \cite[Theorem 8.2.3]{cox}),  and   ${\mathbf K}_X+\Delta\le {\mathbf K}_X+D_0+A$, 
  the assumption that ${\mathbf K}_X+\Delta$ is big  
 implies that $A$ is also big.   
 
Let $\epsilon=\frac13$.  By Theorem \ref{toric},  there exist  a proper Zariski closed subset $Z$ of $X$ and  a positive integer $\ell_0$  
such that, for any set $ \mathcal R$ of $(\Delta, S)$ integral-points  with multiplicity at least $\ell_0$  along $D_0 $, we have  
\begin{align}\label{gabctoric2}
N_{A,S}^{(1)}(P)\ge (1-\epsilon )\cdot h_A(P)-O(1)
\end{align}
holds for all $P\in \mathcal R\setminus Z$ if $\ell\ge \ell_0$. 
 On the other hand, it follows from Definition \ref{campdefi} that 
 \begin{align}\label{trucation}
N_{A,S}^{(1)}(P)\le  \frac12 N_{A,S}(P)\le  \frac12 h_A(P).
\end{align}
 Combining \eqref{trucation} and \eqref{gabctoric2}, we have 
$h_A(P)\le O(1)$
for 	all	$P\in \mathcal R\setminus Z$ if $\ell\ge \ell_0$ as wanted.  
\end{proof}

\begin{proof}[Proof of Theorem \ref{vojtaconj}]
 Let $ \pi:   Y\to X$ be a finite morphism.  Let $D$ be the support of $ H$.
Following the arguments from \cite[Lemma 1]{CZ2013}, we have
	$ K_{Y}\sim  \pi^*(K_{X})+{\rm Ram}$, where ${\rm Ram}$ is the ramification divisor of $\pi$.
	Furthermore, ${\rm Ram}=R+R_D$, where $R_D$ is the contribution coming from the support contained in $D$, i.e.
	$ H =D+R_D$. 
	Since a canonical divisor ${\mathbf K}_X$ can naturally be taken  as $-D_0$, we obtain 
	\begin{equation}\label{ramifidivisor}
		R\sim  D + {\mathbf K}_{Y}\ge\Delta+{\mathbf K}_{Y}. 
	\end{equation}
	Since  $(Y,\Delta)$ is   of  general type, then ${\mathbf K}_Y+\Delta$ is big and hence $R$ is big as well.

 Let $R_0$ be an irreducible component of $R$, which we may assume to be defined over $K$ (see the proof of  \cite[Theorem 3]{GSW2022}).    Without loss of generality, we let $R_0=R$.  Otherwise, we simply repeat the following steps for each irreducible component.  Let $A=\pi(R)$, which is a normal crossings divisor on $(X,\mathbb G_m^n)$ by assumption.
Since $  \pi^*A$ has multiplicity at least 2 along $R$, we have
$$
2\lambda_{R,\q}(P)\le \lambda_{\pi^*A,\q}(P)= \lambda_{A,\q}(\pi(P))$$
for all $\q\in C$ if $P\in Y(K)\setminus R$.  Hence,
\begin{align}\label{coutingzero}
N_{R,S}(P)\le N_{A,S}(\pi(P))-N^{(1)}_{A,S}(\pi(P)),
\end{align}
 if $P\in Y(K)\setminus R$.
 
Since $(Y,\Delta)$ is a Campana orbifold of general type with ${\rm Supp}(\Delta)= {\rm Supp}(H)$ and multiplicity at least $\ell $ and  $ \mathcal R\subset Y(K)$ is a subset of $(\Delta, S)$-integral points, it is clear from Definition \ref{model} that $\pi( \mathcal R)\subset X(K)$ is a subset of $(\Delta',S)$-integral points, where $\Delta'=\pi(\Delta)$ (as $\mathbb Q$ divisors). 

Let $\epsilon=\frac13$.  We now apply Theorem \ref{toric} to the right hand side of  \eqref{coutingzero}, and use  Proposition \ref{bigdivisor} together with the functorial properties of local Weil functions to obtain a sufficiently large integer $\ell_0$ and a proper Zariski closed subset $Z$ of $Y$  such that either $ h_R(P)\le O(1)$ or  
 \begin{align}\label{coutingzero1}
N_{R,S}(P)\le \epsilon h_R(P), \quad\text{and }
\end{align}
\begin{align}\label{coutingzero2}
N^{(1)}_{A,S}(\pi(P))\ge  h_A(\pi(P))-\epsilon  h_R(P)-O(1)
\end{align}	
 for all $P\in  \mathcal R\setminus Z$ if $\ell\ge \ell_0$. 
 On the other hand, by functorial properties,  $ R\le   \pi^* (A)$ (as divisors)  implies that 
\begin{align}\label{Fproxi2}
  m_{R,S}(P)
  &=m_{A,S}(\pi(P))  +O(1)\cr
  &= h_A(\pi(P))-N_{A,S}(\pi(P)) +O(1)\cr
  &\le \epsilon  h_R(P)+ O(1),  
\end{align}
where the last inequality follows from \eqref{coutingzero2}.
Together with \eqref{coutingzero1} and from the fact that $\epsilon=\frac13$, we have $h_R(P)\le O(1)$ 
   if $P\notin  {\rm Supp} (  \pi^*(A))\cup\pi^{-1}(Z)$, as wanted.  
 \end{proof}


\begin{thebibliography}{10}


\bibitem{BM}\textsc{W.~D.~Brownawell and D.~W.~Masser,} \textit{Vanishing
sums in function fields}, Math. Proc. Cambridge Phil. Soc. \textbf{100}
(1986), no. 3, 427--434.

 
 

\bibitem{CT} \textsc{L.~Capuano and A.~Turchet,} \emph{Lang-Vojta conjecture over function fields for surfaces dominating $\mathbb G_m^2$},  Eur. J. Math.   \textbf{8} (2022),  no. 2, 573--610.

 
\bibitem{CZ2008} \textsc{P.~Corvaja and U.~Zannier,} \emph{Some
cases of Vojta's conjecture on integral points over function fields},
J. Algebraic Geom.  \textbf{17} (2008), no. 2, 295--333.

 \bibitem{CZ2013} \textsc{P.~Corvaja and U.~Zannier,} \emph{Algebraic hyperbolicity of ramified covers of $\mathbb G_m^2$ (and integral points on affine subsets of $\PP^2$), } J. Differential Geom.  \textbf{93}  (2013), no.~3, 355--377. 

 \bibitem{cox}   \textsc{D.~A.~Cox,  J.~B.~Little, and H.~K.~Schenck,}   \emph{Toric Varieties}, Graduate Studies in Mathematics, 124. American Mathematical Society, Providence, RI, 2011.
  
  
 \bibitem{Fulton} \textsc{W.~Fulton} \emph{Introduction to Toric Varieties}, Ann. of Math. Stud., 131, Princeton University Press, Princeton, NJ, 1993, xii+157 pp.
 
 \bibitem{GNSW} \textsc{J.~Guo, K.~D.~Nguyen, C.-L.~Sun and J.~T.-Y. Wang,} \emph{Vojta's abc Conjecture  for algebraic tori and applications over function fields}, Adv. Math., to appear. arXiv:2106.15881v3.

 \bibitem{GSW2022} \textsc{J.~Guo, C.-L.~Sun and J.~T.-Y. Wang,} \emph{On Pisot's $d$-th root conjecture for function fields  and related GCD estimates}, J. of Number Theory, \textbf{231} (2022), 401--432. 
 
  
\bibitem{HS} \textsc{M.~Hindry and J.~H.~Silverman}, \emph{Diophantine Geometry. An Introduction.}, Graduate Texts in Mathematics, 201. Springer-Verlag, New York, 2000.

 
\bibitem{LangDG} \textsc{S.~Lang}, \emph{Fundamentals of Diophantine Geometry}, Springer-Verlag, New York, 1983.

  
\bibitem{Levin:GCD} \textsc{A.~Levin}, \emph{Greatest common divisors and
{V}ojta's conjecture for blowups of algebraic tori}, Invent. Math.
\textbf{215} (2019), no. 2, 493--533.

\bibitem{PTV2021} \textsc{M.~Pieropan, A.~Smeets, S.~Tanimoto and A.~V\'arilly-Alvarado}, \emph{Campana points of bounded height on vector group compactifications}, Proc. Lond. Math. Soc. \textbf{123} (2021), no. 3, 57--101.

  
\bibitem{Tur}  \textsc{A.~Turchet}, \emph{Fibered threefolds and Lang-Vojta's conjecture for function fields}, Trans. Amer. Math. Soc.  \textbf{369} (2017), no. 12, 8537--8558.

\bibitem{Vojta}
\textsc{P.~Vojta,} \emph{Diophantine Approximation and Nevanlinna Theory}, {\it Arithmetic Geometry,} 111-224, Lecture Notes in Mathematics {\bf 2009}, Springer-Verlag, Berlin, 2011.

\end{thebibliography}
\end{document}